\documentclass[12pt,a4paper]{article} 
\usepackage[latin1]{inputenc}
\usepackage[matrix,arrow,ps,color,line,curve,frame,all]{xy} 
\SelectTips{cm}{}
\usepackage[english]{babel}
\usepackage[pdftex]{graphicx}
\usepackage{amsfonts}
\usepackage{amssymb}
\usepackage{amsmath}
\usepackage{amsthm}
\usepackage[left=3cm,top=3cm,right=3cm,bottom=3cm]{geometry}
\usepackage{parskip}
\usepackage{verbatim}
\usepackage{dsfont}
\usepackage{ stmaryrd }

\newtheoremstyle{standard}{10pt}{3pt}{\itshape}{}{\bfseries}{.}{.5em}{}
\theoremstyle{standard}
\newtheorem{lemma}{Lemma}[section]
\newtheorem{prop}[lemma]{Proposition}
\newtheorem{thm}[lemma]{Theorem}
\newtheorem{cor}[lemma]{Corollary}

\newtheoremstyle{definition}{10pt}{3pt}{}{}{\bfseries}{.}{.5em}{}
\theoremstyle{definition}
\newtheorem{defi}[lemma]{Definition}  
\newtheorem{nota}[lemma]{Notation}  
\newtheorem{ex}[lemma]{Example}
\newtheorem{expl}[lemma]{Explanation}  
\newtheorem{rem}[lemma]{Remark}

\DeclareMathOperator{\Q}{\mathbf{Qcoh}}

\DeclareMathOperator{\C}{\mathbf{Coh}}
\DeclareMathOperator{\M}{\mathbf{Mod}}
\DeclareMathOperator{\grM}{\mathbf{grMod}}

\DeclareMathOperator{\Alg}{\mathbf{Alg}}

\DeclareMathOperator{\Set}{\mathbf{Set}}

\DeclareMathOperator{\Spec}{Spec}

\DeclareMathOperator{\End}{End}

\DeclareMathOperator{\colim}{colim}

\DeclareMathOperator{\Sym}{Sym}
\DeclareMathOperator{\sym}{\mathfrak{S}}
\DeclareMathOperator{\op}{op}
\DeclareMathOperator{\Hom}{Hom}

\renewcommand{\O}{\mathcal{O}}

\newcommand{\A}{\mathcal{A}}

\renewcommand{\L}{\mathcal{L}}

\renewcommand{\P}{\mathds{P}}

\begin{document}

\title{Tensorial schemes}
\author{Martin Brandenburg\footnote{Fachbereich Mathematik, WWU Münster. E-Mail: brandenburg [at] uni-muenster [dot] de}}
\date{}
\maketitle

\begin{abstract} \noindent Jacob Lurie (\cite{Lur}) has shown that for geometric stacks $X,Y$ every cocontinuous tensor functor $F : \Q(X) \to \Q(Y)$ is the pullback $f^*$ of a morphism $f : Y \to X$ under the additional assumption that $F$ is \emph{tame}. In this note we get rid of this assumption if $X$ is a projective scheme. In general, we call a scheme $X$ \emph{tensorial} if every cocontinuous tensor functor $\Q(X) \to \Q(Y)$ is induced by a morphism $Y \to X$ and show that projective schemes are tensorial and tensorial schemes are closed under various operations. \end{abstract}

\section*{Introduction}
 
Gabriels Reconstruction Theorem (\cite{Gab}) states that a noetherian scheme $X$ can be reconstructed from the abelian category of quasi-coherent modules $\Q(X)$. Meanwhile there are many variants of this Theorem (\cite{Bal}, \cite{Gar}), but they do not recover $X$ in a functorial way. Recently, Jacob Lurie (\cite{Lur}) has shown that a geometric stack (in particular every quasi-compact semi-separated scheme) can be functorially reconstructed from $\Q(X)$, considered as an abelian tensor category. Namely, he has shown that for geometric stacks $X,Y$ the category of morphisms $Y \to X$ is equivalent to the category of \emph{tame} cocontinuous tensor functors $\Q(X) \to \Q(Y)$. Here, tameness is a rather global flatness condition. Our goal is to eliminate this tameness condition: We study schemes $X$ such that for all schemes $Y$ (and therefore for all algebraic stacks $Y$) the category of morphisms $Y \to X$ is equivalent to the category of cocontinuous tensor functors $\Q(X) \to \Q(Y)$ and call them \emph{tensorial}. Our main result is that every projective scheme (over an affine base) is tensorial.
  
The first section covers some basics on tensor categories (a.k.a. monoidal categories) and our conventions. The second section reviews the universal cocompletion of a (tensor) category. The focus lies on the two examples $\M(S)$, where $S$ is a ring, and $\grM(S)$, where $S$ is a graded ring, preparing the proof that affine and projective schemes are tensorial. The third section is devoted to the proof of a universal property of $\Q(\P^n)$, which says that this is the free cocomplete tensor category on an invertible object $\L$ and a "good" epimorphism $1^{n+1} \to \L$. This is a categorification of the universal property of $\P^n$ and will show in the fifth section that $\P^n$ are tensorial. The fourth section categorifies the universal property of a closed immersion, or more generally an affine morphism. This will imply that closed subschemes of tensorial schemes are tensorial.
 

For various suggestions I would like to thank James Dolan, Todd Trimble, Laurent Moret-Bailly, Jacob Lurie, Tom Goodwillie and Christopher Deninger.
 
\section{Tensor categories} Throughout this note, every ring in consideration is commutative. A \emph{tensor category} is a category together with a tensor product which is unital, associative and symmetric up to compatible isomorphisms; these are called $\otimes$-categories ACU in (\cite{SaR},2.4). Additionally, we assume them to be $R$-linear for some fixed ring $R$: This means that the underlying category is $R$-linear and the tensor product is $R$-linear in both variables (\cite{SaR}, 0.1.2). Tensor functors are understood to be strong, that is they respect the tensor structure up to a canonical \emph{iso}morphism (as in \cite{SaR}, 4.1.1; 4.2.4). Besides they should be, of course, $R$-linear (\cite{SaR}, 4.1.3). For $R$-linear tensor categories $C,D$ we denote by $\Hom_{\otimes/R}(C,D)$ the category of all morphisms $C \to D$, or just $\Hom_{\otimes}(C,D)$ if $R$ is clear from the context. Morphisms in this category are tensor natural transformations, that is natural transformations which are compatible with the tensor structure (\cite{SaR}, 4.4.1). The unit of a tensor category $C$ is usually denoted by $1_C$. The a priori noncommutative $R$-algebra $\End(1_C)$ turns out to be commutative by a variation of the Eckmann-Hilton argument (\cite{SaR}, 1.3.3.1).
  
By a \emph{cocomplete} tensor category we mean a tensor category whose underlying category is cocomplete (i.e. has all small colimits) such that the tensor product is cocontinuous in each variable. This means that for all objects $X$ and all small diagrams $\{Y_i\}$ the canonical morphism
\[\colim_i (X \otimes Y_i) \to X \otimes \colim_i Y_i\]
is an isomorphism; similarily for the other variable, which also follows by symmetry. For discrete diagrams this is just the categorified distributive law
\[\bigoplus_i (X \otimes Y_i) = X \otimes \bigoplus_i Y_i.\]
Therefore we can think of $R$-linear cocomplete tensor categories as categorified $R$-algebras and might call them $R$-$2$-algebras. In fact, Alex Chirvasitu and Theo Johnson-Freyd (\cite{ChJo}, 2.3.1) call them $2$-rings, dropping the enrichment $\M(R)$ and assuming presentability of the underlying category.
  
If $S$ is an $R$-algebra, then $\M(S)$ is an $R$-linear cocomplete tensor category. The tensor product is the usual tensor product of modules and the unit is $S$. More generally, if $X$ is an $R$-scheme, then its category of quasi-coherent modules $\Q(X)$ is an $R$-linear cocomplete tensor category with tensor product $\otimes_X$ and unit $\O_X$. This is our main example. Tannaka Reconstruction theorems such as the one by Jacob Lurie (\cite{Lur}) suggest that all the information of (a nice) $X$ is already encoded in this $2$-algebra $\Q(X)$ and therefore we can think of usual algebraic geometry as $2$-affine (see also \cite{ChJo}, 1.2).
 
If $C,D$ are $R$-linear cocomplete tensor categories, we denote by $\Hom_{c\otimes/R}(C,D)$ the category of all cocontinuous tensor functors $C \to D$; if $C,D$ are just cocomplete categories, we denote by $\Hom_{c}(C,D)$ the category of cocontinuous functors. For example, every morphism of $R$-schemes $f : X \to Y$ induces a cocontinuous tensor functor $f^* : \Q(Y) \to \Q(X)$. Our main question is: Does this induce an equivalence of categories?
\begin{eqnarray*}
\Hom_R(X,Y) & \to & \Hom_{c\otimes/R}(\Q(Y),\Q(X)) \\ f & \mapsto & f^* \end{eqnarray*}
Remark that this functor is automatically faithful since $\Hom(X,Y)$ is discrete.
 
\section{Universal cocompletion}

In this section we review a basic construction from category theory, namely the universal cocompletion of a small category. I am indebpted to James Dolan who suggested this approach to me in order to achieve the universal property of $\Q(X)$ for affine or projective $X$. My previous proof was rather ad hoc, since it used locally free resolutions and cohomological methods.
 
Let $C$ be a small category. Let $\widehat{C} = \Hom(C^{\op},\Set)$ be the category of presheaves on $C$. This category is cocomplete and comes equipped with the Yoneda embedding $y : C \to \widehat{C}$. The following proposition says that it is the universal cocompletion:
 
\begin{prop} \label{uEc} Let $D$ be a cocomplete category. With the above notations, $y$ induces an equivalence of categories $\Hom(C,D) \cong \Hom_c(\widehat{C},D)$.\end{prop}
 
\begin{proof} This is well-known (\cite{Kel}, 4.4), we sketch the proof. We map a cocontinuous functor $\widehat{C} \to D$ to the composition $C \stackrel{y}{\to} \widehat{C} \to D$. In the other direction, let $\alpha : C \to D$ be a functor. This induces a cocontinuous functor $D \to \widehat{C} \to \widehat{C}$ which has a left adjoint $\hat{\alpha} : \widehat{C} \to D$, which is given explicitely as follows: Every presheaf $F : C^{\op} \to \Set$ can be canonically written as a colimit representables:
\[F = \colim_{e \in F(c)} \Hom(-,c)\]
This is an easy consequence of the Yoneda-Lemma. Now we have to define
\[\widehat{\alpha}(F) = \colim_{e \in F(c)} \alpha(c). \qedhere\]
\end{proof}
 
The same construction yields the universal cocompletion of a small tensor category:

\begin{prop} Let $C$ be a small tensor category. Then we can endow $\widehat{C}$ with the structure of a cocomplete tensor category (in an essentially unique way), such that the Yoneda embedding $y : C \to \widehat{C}$ is a tensor functor. Moreover, for every cocomplete tensor category $D$ we have $\Hom_{\otimes}(C,D) \cong \Hom_{c\otimes}(\widehat{C},D)$. \end{prop}

\begin{proof} This is also well-known (\cite{ImKe}). We have to extend the tensor product from $C$ to $\widehat{C}$, the unit being $y(1_C)$. Guided by the above representation of a presheaf as a colimit of representables, we have to define for all $F,G : C^{\op} \to \Set$
\[F \otimes G := \colim_{(e,e') \in F(c) \times G(c')} \Hom(-,c \otimes c')\]
Then $y$ becomes a tensor functor and it is easy to check that $\Hom(C,D) \cong \Hom_{c\otimes}(\widehat{C},D)$ restricts to the desired equivalence. \end{proof}
 
\begin{rem} The last propositions hold equally well in the context of $R$-linear categories. Here we have to replace $\Set$ by $\M(R)$, i.e. presheaves on $C$ are $R$-linear functors $C^{\op} \to \M(R)$. \end{rem}
  
\begin{ex} Let $S$ be an $R$-algebra. Let $\{S\}$ be the $R$-linear tensor category which consists of just one object $S$ whose endomorphism algebra is defined by $S$. The universal cocompletion identifies with $\widehat{\{S\}} \cong \M(S)$, the category of right-$S$ modules. Its universal property reads as:\end{ex}

\begin{prop} Let $C$ be a cocomplete $R$-linear category and $S$ be an $R$-algebra. Then there is an equivalence of categories
\[\Hom_{c\otimes/R}(\M(S),C) \cong \Hom_{\Alg(R)}(S,\End(1_C)).\]
\end{prop}
  
\begin{proof} By the universal property of the cocompletion, the left hand side identifies with the category of $R$-linear tensor functors $\{S\} \to C$. The object function has to be $S \mapsto 1_C$ (up to equivalence). We are left with a morphism of $R$-algebras $S \to \End(1_C)$. \end{proof}

\begin{rem} The proposition may be read as: Categorification $S \mapsto \M(S)$ is left adjoint to decategorification $C \mapsto \End(1_C)$. \end{rem}

\begin{cor} \label{unimod} Let $S$ be an $R$-algebra and $X$ be an $R$-scheme. Then we have an equivalence of categories
\[\Hom_R(X,\Spec(S)) \cong \Hom_{\Alg(R)}(S,\Gamma(X,\O_X)) \cong \Hom_{c\otimes/R}(\M(S),\Q(X)).\]
The composition is given by associating to every $f : X \to \Spec(S)$ the pullback functor $f^* : \Q(\Spec(S)) \to \Q(X)$.
\end{cor}
 
This Corollary yields a positive answer to our main question in the affine case. In order to extend this to the projective case, we need another example of a cocompletion, which is due to James Dolan:

\begin{ex} Let $G$ be an abelian group and $S = \bigoplus_{g \in G} S_g$ be a $G$-graded $R$-algebra. Consider the following tensor category $C[S]$: Objects are elements of $G$. The morphisms are given by $\Hom_{C[S]}(g,h) := S_{h-g}$. The composition is induced by the multiplication of $S$. The tensor structure is induced by the addition of $G$. Remark that every symmetry $g \otimes g \cong g \otimes g$ is just the identity. A presheaf $C[S]^{\op} \to \M(R)$ is given by a family of $R$-modules $(M_g)_{g \in G}$ together with compatible maps $S_{g-h} \mapsto \Hom(M_g,M_{h})$. This endows the $R$-module $\bigoplus_{g \in G} M_{-g}$ with the structure of a graded $S$-module. We get an equivalence between $\widehat{C[S]}$ and the category of graded modules $\grM(S)$, which maps the representable $\Hom(-,g)$ to the twist $S(g)$, which is defined by $S(g)_h=S_{g+h}$. The universal property becomes:\end{ex}

\begin{prop} With the above notations, for every cocomplete $R$-linear tensor category $C$ the category $Hom_{c\otimes/R}(\grM(S),C)$ is equivalent to the following category: Objects are families of objects $(X_g)_{g \in G}$ in $C$ together with compatible isomorphisms $X_0 \cong 1_C , X_g \otimes X_h \cong X_{g+h}$ and morphisms of $R$-modules $S_g \to \Hom(1_C,X_g)$, which are compatible in the sense that for $g,h$ the diagram
\[\xymatrix{S_g \times S_h \ar[d]^{\cdot} \ar[r] & \Hom(1_C,X_g) \times \Hom(1_C,X_h) \ar[d]^{\otimes} \\ S_{g+h} \ar[r] & \Hom(1_C,X_{g+h})}\]
commutes. Moreover, the symmetries $X_g \otimes X_g \cong X_g \otimes X_g$ are the identity.
\end{prop}
 
\begin{proof} We have $\Hom_{c\otimes/R}(\grM(S),C) \cong \Hom_{c\otimes/R}(\widehat{C[S]},C) \cong \Hom_{\otimes/R}(C[S],C)$. Now a tensor functor $X : C[S] \to C$ is given by a family of $(X_g)$ as above together with compatible morphisms $S_{h-g} \to \Hom(X_g,X_h)$, where the latter identifies with $\Hom(X_g \otimes X_{-g},X_h \otimes X_{-g}) \cong \Hom(1_C,X_{h-g})$.\end{proof}
 
\begin{defi} Let $C$ be a tensor category. An object $\L \in C$ is called \emph{invertible} if the following holds:
\begin{itemize} \item $\L$ has a dual (\cite{KaSh}, 4.2.11) $\L^{\otimes -1}$ such that the counit $\L^{\otimes -1} \otimes \L \to 1_C$ is an isomorphism.
\item The symmetry $\L \otimes \L \cong \L \otimes \L$ equals the identitiy.
\end{itemize}
\end{defi}

\begin{rem} \noindent \begin{enumerate} \item Usually (\cite{SaR}, 2.5) the symmetry is not required to be the identity. But this is useful for our purposes, as the previous proposition shows. Another reason is that for every locally free sheaf $\L$ of rank $1$ on a scheme the symmetry $\mathcal{L} \otimes \mathcal{L} \cong \mathcal{L} \otimes \mathcal{L}$ \emph{is} the identity.
 \item If $\L$ is an object of $C$ such that there is some object $\L^{\otimes -1}$ and an isomorphism $c : \L^{\otimes -1} \otimes \L \cong 1_C$, then there is a unique morphism $e : 1_C \to \L \otimes \L^{\otimes -1}$ such that $\L$ has $\L^{\otimes -1}$ as a dual with counit $c$ and unit $e$ (\cite{SaR}, 2.5.5.2). Besides, also $e$ is an isomorphism.
\item If $\L$ is invertible, then $\L^{\otimes -1}$ is also invertible. In fact, two applications of (\cite{Kas}, XIV.3.1) show that the symmetry on $\L^{\otimes -1} \otimes \L^{\otimes -1}$ is the identity.
\item If $\L$ is invertible, we can define inductively $\L^{\otimes a}$ for every $a \in \mathbb{Z}$.
\item Although the counit and unit belong to the data of an invertible object, they are uniquely determined up to canonical isomorphism. This allows us to treat invertible objects just as special objects.
\item Tensor functors preserve invertible objects.
\end{enumerate}
\end{rem}

\begin{cor} \label{unigr} Let $S=\oplus_{n \in \mathbb{N}} S_n$ be a graded $R$-algebra and $C$ be cocomplete $R$-linear tensor category. Then $\Hom_{c\otimes/R}(\grM(S),C)$ identifies with the category of pairs $(\L,s)$, where $\L \in C$ is invertible and $s : S_1 \to \Hom(1_C,\L)$ is a homomorphism of $R$-modules. \end{cor}
 
This directly follows from the last proposition. An inspection of the proofs shows that that the equivalence maps $F : \grM(S) \to C$ to the pair $(\L,s)$, where $\L:=F(S(1))$ and $s:= S_1 \cong \Hom(S,S(1)) \stackrel{F}{\longrightarrow} \Hom(1_C,\L)$.
 
\section{Projective space}

We already know a universal property of the $2$-$R$-algebra $\Q(\Spec(S))$ which is similar to the universal property of the $R$-scheme $\Spec(S)$ (Corollary \ref{unimod}). Now we try to find a universal property of $\Q(\P^n_R)$ which is similar to the universal property of $\P^n_R$ as a moduli space of line bundles with $n+1$ global generators:
\[\Hom_R(X,\P^n_R) \cong \{(\L,s) : \L \text{ invertible sheaf on } X,~ s : \O_X^{n+1} \twoheadrightarrow \L\}\]
The proof will be essentially a combination of Corollary \ref{unigr} and Serre's classical results on coherent sheaves on projective space.
  
\begin{thm} \label{unipr} The tensor category $\Q(\P^n_R)$ has the following universal property: If $C$ is a cocomplete $R$-linear tensor category, then $\Hom_{c\otimes/R}(\Q(\P^n_R),C)$ is equivalent to the category of pairs $(\L,s)$, where $\L \in C$ is invertible and $s : 1^{n+1} \to \L$ is a morphism such that the sequence
\[\hspace{1.5cm} \xymatrix@C=40pt{(\L^{\otimes -1})^{\binom{n+1}{2}} \ar[r]^-{r} &  1^{n+1} \ar[r]^-{s} &  \L \ar[r] &  0} \hspace{1cm} (*)\]
is a cokernel diagram. Here $F : \Q(\P^n_R) \to C$ corresponds to the pair $(\L,s)$ with $\L := F(\O_{\P}(1))$ and $s := F(\O_{\P}^{n+1} \twoheadrightarrow \O_{\P}(1))$. \end{thm}
 
\begin{expl} The morphism $r$ is defined as follows: It corresponds to a family of morphisms $\L^{\otimes -1} \to 1$ and thus morphisms $r_{i,j,k} : 1 \to \L$ for $0 \leq i < j \leq n$ and $0 \leq k \leq n$, which we define by $r_{i,j,i} = -s_j, ~ r_{i,j,j} = s_i$ and $r_{i,j,k} = 0$ for $k \neq i,j$. There is a more suggestive notation for this morphism which is motivated from the example $C=\Q(X)$: Imagine $(\L^{\otimes -1})^{\binom{n+1}{2}}$ as a twisted free module with basis $\{e_{i,j}\}_{i<j}$ and $1^{n+1}$ as a free module with basis $\{e_i\}_{i}$. Also, $s_i : 1 \to \L$ may be viewed as a global section of $\L$. Then $r$ just maps $e_{i,j} \mapsto s_i e_j - s_j e_i$. In the following we will often use this kind of suggestive notation. The requirement that $s$ is a cokernel of $r$ asserts in a sense that $\L$ is generated by the global sections $s_0,\dotsc,s_n$ and essentially the only relations between them are $s_i \otimes s_j = s_j \otimes s_i : 1 \to \L^{\otimes 2}$. This might be called a \emph{good} epimorphism $s$.
\end{expl}
 
\begin{rem} In general, not every epimorphism $1^{n+1} \to \L$ is good. The free $R$-linear cocomplete tensor category on an epimorphism $1 \to 1$ provides a counterexample, which was shown to me by James Dolan. It is explicitely given by the category of $R[x]$-modules on which $x$ acts regular. Here $x : R[x] \to R[x]$ is an epimorphism, which is not an isomorphism and therefore not good. However, the following lemma shows that in $\Q(X)$ this phenomenon does not arise, thereby establishing the analogy between the universal properties of $\Q(\P^n)$ and $\P^n$. \end{rem}
  
\begin{lemma} \label{lempr} Let $X$ be a scheme. Then $\L \in \Q(X)$ is invertible in the above sense precisely when it is in the usual sense, i.e. locally free of rank $1$. In this case, for every epimorphism $s : \O_X^{n+1} \to \L$ the sequence $(\L^{\otimes -1})^{\binom{n+1}{2}} \to \O_X^{n+1} \to \L \to 0$ is a cokernel diagram. \end{lemma}

\begin{proof} If $\L$ is locally free of rank $1$, then it is clear that $\L$ is invertible. Now assume that $\L$ is invertible. Then also every restriction is invertible. Thus we can work locally on $X$. Write a generator of $\L \otimes \L^{\otimes -1} \cong \O_X$ as a finite sum of pure tensors. The sections appearing from $\L$ generate it locally. Thus $\L$ is locally of finite type. Now for every $x \in X$, we have $\L(x) \otimes_{\kappa(x)} \L^{\otimes -1}(x) \cong \kappa(x)$, where $\L(x) := \L_x \otimes_{\O_{X,x}} \kappa(x)$. Linear algebra tells us $\L(x) \cong \kappa(x)$. Nakayama's Lemma shows that $\L_x$ is generated by a single element. Since $\L$ is locally of finite type, this is true in an open neighborhood of $x$. The same works for $\L^{\otimes -1}$ instead of $\L$. The local generators cannot have any torsion since their tensor product generates $\L \otimes \L^{\otimes -1} \cong \O_X$. Thus they are free generators and we see that $\L$ is locally free of rank $1$.
 
In order to check the exactness of the sequence, remark that $X$ is covered by the open subsets $D(s_i)$ on which $s_i$ generates $\L$. Thus we may assume that $\L$ is free generated by $s_0$. In fact, we may assume $X = \Spec(S), \L = \widetilde{S}$ and $s_0 = 1$ and have to show the exactness of
\[\xymatrix@C=70pt{S^{\binom{n+1}{2}} \ar[r]^{e_{i,j} \mapsto s_i e_j - s_j e_i} &  S^{n+1} \ar[r]^{e_i \mapsto s_i} &  S \ar[r] &  0}\]
in $\M(S)$. It is a complex since $s_i s_j = s_j s_i$ holds in $S$. Now let $(x_0,\dotsc,x_n) \in S^{n+1}$ be in the kernel, i.e. $x_0 = - s_1 x_1 - \dotsc - s_n x_n$. Then
\[(x_0,\dotsc,x_n) = x_1 (-s_1,1,0,\dotsc) + \dotsc + x_n (-s_n,\dotsc,0,1)\]
is the image of $x_1 e_{0,1} + \dotsc + x_n e_{0,n}$. \end{proof}
 
We prepare us for the proof of Theorem \ref{unipr}. Fix a cocomplete $R$-linear tensor category $C$.
 
\begin{defi} Let $X \in C$ and $m \geq 0$. Then the symmetric group $\sym(m)$ acts on $X^{\otimes m}$. The object of coinvariants, that is the coequalizer of all the $m!$ maps $X^{\otimes m} \to X^{\otimes m}$ is denoted by $\Sym^m(X)$. \end{defi}
 
\begin{lemma} \label{symsum} For $X,Y \in C$ we have ${\Sym}^m(X \oplus Y) \cong \bigoplus_{p+q=m} {\Sym}^p(X) \otimes {\Sym}^q(Y)$. \end{lemma}

\begin{proof} Let $Z \in C$. A morphism ${\Sym}^m(X \oplus Y) \to Z$ corresponds to a morphism on $(X \oplus Y)^{\otimes m}$, which is invariant under permutation of the factors. Now $(X \oplus Y)^{\otimes m}$ is the direct sum of all $m$-fold tensor products consisting of $X$ or $Y$. Concerning invariant morphisms on such a direct sum, we only have to consider the summands $X^{\otimes p} \otimes Y^{\otimes q}$ with $p+q=n$ and ensure invariance separately on $X^{\otimes p}$ and $Y^{\otimes q}$. Thus we are left with a morphism $\bigoplus_{p+q=m} {\Sym}^p(X) \otimes {\Sym}^q(Y) \to Z$. By Yoneda's Lemma, we are done. \end{proof}

\begin{lemma} \label{symex} For every cokernel diagram
\[\xymatrix{X \ar[r] &  Y \ar[r]&  Z \ar[r]&  0}\]
we get a cokernel diagram
\[\xymatrix{X \otimes {\Sym}^{m-1}(Y)  \ar[r] & {\Sym}^m(Y) \ar[r] & {\Sym}^m(Z) \ar[r]&  0.}\]
\end{lemma}
 
Here the left morphism is induced by $X \otimes Y^{\otimes (m-1)} \to Y^{\otimes m}$.
 
\begin{proof} Let $T \in C$. Then morphisms $\Sym^m(Z) \to T$ correspond to invariant morphisms $Z^{\otimes m} \to T$. But morphisms $Z^{\otimes m} \to T$ correspond to morphisms $Y^{\otimes m} \to T$ with the property that for every $1 \leq i \leq m$ the composition with $Y \otimes \dotsc \otimes X \otimes \dotsc \otimes Y \to Y^{\otimes m}$ is trivial, where $X \to Y$ is the $i$th factor. This identification respects invariance of the morphisms, since $Y \to Z$ and therefore $Y^{\otimes m} \to Z^{\otimes m}$ is an epimorphism. But then invariance reduces the condition to the special case $i=1$. Thus we are left with a morphism $\Sym^m(Y) \to T$ which vanishes on $X \otimes {\Sym}^{m-1}(Y)$. \end{proof}
 
\begin{nota} In the following let $S = R[x_0,\dotsc,x_n]$ as graded $R$-algebra. Denote by $H_m$ the set of (monic) monomials of degree $m$ in $x_0,\dotsc,x_n$. Thus $H_m$ is a basis of the $R$-module $S_m$ of homogeneous polynomials of degree $m$. For $m<0$ we have $H_m = \emptyset$. \end{nota}
  
\begin{lemma} \label{seq} Given a pair $(\L,s)$, where $\L \in C$ is invertible and $s : 1^{n+1} \to \L$ is a morphism, such that
\[\xymatrix{(\L^{\otimes -1})^{\binom{n+1}{2}} \ar[r] &  1^{n+1} \ar[r] &  \L \ar[r] &  0}\]
is a cokernel diagram. Then for all $m \geq 0$ the same is true for
\[\xymatrix{(\L^{\otimes -1})^{H_{m-1} {\binom{n+1}{2}}} \ar[r] &  1^{H_m} \ar[r] &  \L^{\otimes m} \ar[r] &  0.}\]
\end{lemma}
 
Here the middle morphism is given by $e_p \mapsto p[x_i \mapsto s_i]$ for $p \in H_m$ (suggestive notation), i.e. we replace $x_i$ by $s_i$ and products by tensor products in the polynomial $p$. The left morphism is given by $e_{q,i,j} \mapsto s_i e_{x_j q} - s_j e_{x_i q}$, where $q \in H_{m-1}$ und $0 \leq i < j \leq n$.
 
\begin{proof} According to our definition of an invertible object, $\sym(m)$ acts trivially on $\L^{\otimes m}$, thus $\Sym^m(\L) = \L^{\otimes m}$. The generalization of Lemma \ref{symsum} to finitely many summands yields  $\Sym^m(1^{n+1}) \cong 1^{H_m}$; the same for $m-1$. Now the claim follows from Lemma \ref{symex}. \end{proof}
 
\begin{lemma} \label{presab}  Let $a,d \in \mathbb{Z}$ with $a+d \geq 0$. Then we have an exact sequence of graded $S$-modules
\[\xymatrix{S(-a-1)^{H_{a+d-1} \binom{n+1}{2}} \ar[r] & S(-a)^{H_{a+d}} \ar[r] & S(d)_{\geq a} \ar[r] & 0.}\]
\end{lemma}
 
The middle morphism is defined by $e_p \mapsto p$ and the left one by $e_{q,i,j} \mapsto x_i e_{x_j q} - x_j e_{x_i q}$, where $0 \leq i<j \leq n$ and $q \in H_{a+d-1}$.

\begin{proof} For a graded $S$-module $M$ we have to prove the exactness of
\[\xymatrix@C=10pt{0 \ar[r] & \Hom(S(d)_{\geq a},M) \ar[r] & \Hom(S(-a)^{H_{a+d}},M) \ar[r] & \Hom(S(-a-1)^{H_{a+d-1} \binom{n+1}{2}},M).}\]
The map on the right identifies with
\[{M_a}^{H_{a+d}} \to {M_{a+1}}^{H_{a+d-1} \binom{n+1}{2}} ~,~ (m_p)_{p \in H_{a+d}} \mapsto (x_i m_{x_j q} - x_j m_{x_i q})_{q,i,j}\]
and the left one maps $\phi : S(d)_{\geq a} \to M$ to $(\phi(p))_{p \in H_{a+d}}$. Thus we have to show the following: For every tuple $(m_p)$ with $x_i m_{x_j q} = x_j m_{x_i q}$ there is a unique $S$-linear graded homomorphism $\phi : S(d)_{\geq a} \to M$ with $\phi(p)=m_p$ for all $p$.
 
We construct the $R$-linear maps $\phi_k : (S(d)_{\geq a})_k \to M_k$ recursively: For $k<a$ we have $\phi_k=0$. Define $\phi_a : S_{a+d} \to M_a$ to be linear extension of $p \mapsto m_p$ for $p \in H_{a+d}$. Now assume $k \geq a$ and $\phi_k$ is already defined. Define $\phi_{k+1} : S_{k+d+1} \to M_{k+1}$ to be the linear extension of the following map $H_{k+d+1} \to M_{k+1}$: Every monomial of degree $k+d+1$ may be written as $x_s^{v_s} \cdots x_n^{v_n}$ for some unique $0 \leq s \leq n$ with $v_s \geq 1$, since otherwise the degree would be $0=k+d+1 \geq a+d+1 \geq 1$. Now let
\[\phi_{k+1}(x_s^{v_s} \cdots x_n^{v_n}) := x_s \phi_k(x_s^{v_s-1} \cdots x_n^{v_n}).\]
Then $\phi = \oplus_k \phi_k$ is $R$-linear and satisfies $\phi(p)=m_p$ for all $p \in H_{a+d}$. In order to show that $\phi$ is $S$-linear, it suffices to prove
\[\phi_{k+1}(x_i q) = x_i \phi_k(q)\]
for all $k \geq a$, $0 \leq i \leq n$ and $q \in H_{k+d}$. We use induction on $k$. Write as above $q = x_s^{v_s} \cdots x_n^{v_n}$ with $v_s \geq 1$. For $i \leq s$ the equation follows immediately from the definitions. For $i \geq s$, we have
\[\phi_{k+1}(x_i q) = x_s \phi_k(x_i x_s^{v_s-1} \cdots x_n^{v_n}).\]
Assuming the claim holds for $k-1$ and $k > a$, we get, as desired,
\[x_s \phi_k(x_i x_s^{v_s-1} \cdots x_n^{v_n}) = x_s x_i \phi_k(x_s^{v_s-1} \cdots x_n^{v_n}) = x_i \phi_k(q).\]
The base case $k=a$ follows from the relations of the elements $m_p \in M_a$:
\[x_s \phi_a(x_i x_s^{v_s-1} \cdots x_n^{v_n}) = x_s m_{x_i x_s^{v_s-1} \cdots x_n^{v_n}} = x_i m_q = x_i \phi_a(q).\]
The uniqueness of $\phi$ is clear: Every step of the definition above was \emph{forced} by the condition that $\phi$ is $S$-linear and maps $p \mapsto m_p$. \end{proof}
 
\begin{lemma} \label{ind} Let $X$ be a concentrated (i.e. quasi-compact and quasi-separated) $R$-scheme. Denote by $\Q_f(X)$ the category of quasi-coherent sheaves of finite presentation. Then for every cocomplete $R$-linear category $C$, we have an equivalence of categories
\[\Hom_{c\otimes/R}(\Q(X),C) \cong \Hom_{fc\otimes/R}(\Q_f(X),C), ~ F \mapsto F|_{Q_f(X)}.\]
\end{lemma}
Here $\Hom_{fc\otimes/R}$ means the category of $R$-linear tensor functors which are finitely cocontinuous, or equivalently, preserve cokernels (since direct sums are biproducts and preserved due to linearity).
 
\begin{proof} The idea is that every quasi-coherent sheaf is the colimit of all quasi-coherent sheaves of finite presentation mapping into it. Thus, for $G \in \Hom_{fc\otimes}(\Q_f(X),C)$ we have to define its extension $F \in \Hom_{c\otimes}(\Q(X),C)$ by
\[F(M) = \colim_{N \to M, N \in \Q_f(X)} G(N).\]
This works since $\Q_f(X)$ is essentially small. The action of $F$ on morphisms is clear. Now it is easy to see that $F$ is $R$-linear, extends $G$ and preserves directed unions of submodules. Hence, $F$ preserves arbitrary direct sums. To see that $F$ preserves cokernels, it is enough to write every homomorphism $M \to M'$ in $\Q(X)$ as a directed colimit of homomorphisms $N_i \to N'_i$, where $N_i,N'_i \in \Q_f(X)$, and then use that $G$ preserves their cokernels. But this follows from (\cite{KaSh}, 6.4.4). Thus $F$ is cocontinuous. We have $F(\O_X)=G(\O_X) \cong 1_C$. For $M,M' \in \Q(X)$, there is a canonical homomorphism $F(M) \otimes F(M') \to F(M \otimes M')$ extending the maps $G(N) \otimes G(N') = G(N \otimes N') \to G(M \otimes M')$ for $N \to M, N' \to M$ of finite presentation. Since it is natural in both $M$ and $M'$ and $F$ is cocontinuous, we may reduce to the case that $M,M'$ are of finite presentation in order to show that it is an isomorphism, in which case it is clear.\end{proof}

Finally:

\begin{proof}[Proof of Theorem \ref{unipr}] First assume that $R$ is a field. In this case Serre determined the category of \emph{coherent} sheaves on $\P^n_R$: Let $D$ be the category consisting of finitely generated graded $S$-modules and equivalence classes of homomorphisms of graded modules; here two homomorphisms are identified if they equal in sufficiently high degrees. Then there is an equivalence $D \cong \C(\P^n_R)$, which associates to every graded module $M$ its associated sheaf $\widetilde{M}$ (\cite{Ser}, III.3, 65, Prop. 5 and Prop. 6). This can be generalized to noetherian rings $R$ (\cite{EGAIII}, Scholie (2.3.3)).
 
Thus there is an equivalence between $\Hom_{fc\otimes}(\C(\P^n_R),C)$ and
\[\{F \in \Hom_{fc\otimes}(\grM_f(S),C) : F(i) \text{ is an isom. for all } i : M_{\geq a} \subseteq M\}.\]
Here $\grM_f(S)$ is the category of finitely generated graded $S$-modules. Now using some Lemma \ref{ind} and the corresponding result for graded modules, we see that $\Hom_{c\otimes}(\Q(\P^n_R),C)$ is equivalent to
\[\{F \in \Hom_{c\otimes}(\grM(S),C) : F(i) \text{ is an isom. for all } i : M_{\geq a} \subseteq M \in \grM_f(S)\}.\]
Since every $M \in \grM_f(S)$ is the cokernel of a homomorphism between finite direct sums of twists $S(d)$, the condition on $F$ reduces to: For all $a,d \in \mathbb{Z}$ the inclusion $S(d)_{\geq a} \subseteq S(d)$ is mapped to an isomorphism. For $a+d < 0$ this is trivial. So let $a+d \geq 0$. Using Lemma \ref{presab} for $\L := F(S(1))$, we see that the condition says that
\[\xymatrix{(\L^{\otimes (-a-1)})^{H_{a+d-1} \binom{n+1}{2}} \ar[r] & (\L^{\otimes -a})^{H_{a+d}} \ar[r] & \L^{\otimes d} \ar[r] & 0}\]
is a cokernel diagram. Let es denote this sequence by $P(a,d)$. Then $P(a,d)$ is just $P(0,a+d)$ tensored with $\L^{\otimes -a}$. Besides, $P(0,1)$ is the already known sequence
\[\xymatrix{(\L^{\otimes -1})^{\binom{n+1}{2}} \ar[r] &  1^{n+1} \ar[r] &  \L \ar[r] &  0}\]
and $P(0,m)$ for $m \geq 0$ is the implied sequence from Lemma \ref{seq}. Thus the condition on $F$ just says that $P(0,1)$ is a cokernel diagram. The universal property of $\grM(S)$ (Corollary \ref{unigr}) now finishes the proof. \end{proof}
 
Now we want to get rid of the assumption that $R$ is noetherian. In fact, by a base change argument, we will reduce to the noetherian ring $\mathds{Z}$.
 
Let $R \to S$ be a homomorphism of rings. Then there is a forgetful $2$-functor from $S$-linear cocomplete tensor categories to $R$-linear ones, and this has a $2$-left adjoint: Let $C$ be a $R$-linear cocomplete tensor category. Define $C_S$ to be the category of $S$-left modules in $C$, that is $X \in C$ together with a homomorphism of $R$-algebras $S \to \End(X)$. This is $S$-linear. For $X,Y \in C_S$ define $X \otimes_S Y$ to be the quotient of $X \otimes Y$ which identifies the two actions of $S$. The forgetful functor $C_S \to C$ creates colimits and has a left adjoint $p : C \to C_S, X \mapsto S \otimes_R X$. Put $1_{C_S} := p(1_C)$. Then it can be shown that $C_S$ is a $S$-linear cocomplete tensor category and $p$ is a cocontinuous $R$-linear tensor functor, which is universal in the following sense:
 
\begin{prop} For every cocomplete $S$-linear tensor category, $p$ induces an equivalence of categories $\Hom_{c\otimes/S}(C_S,D) \cong \Hom_{c\otimes/R}(C,D)$. \end{prop}
 
For details, see (\cite{SaR}, II.1.5). Now if $X$ is an $R$-scheme and we consider the base change $X_S := X \times_R S$, then it is easily seen that $\Q(X_S) \cong \Q(X)_S$.

\begin{proof}[Reduction of Theorem \ref{unipr} to $R=\mathds{Z}$] Use the base change homomorphism $\mathbb{Z} \to R$:
\[\Hom_{c\otimes/R}(\Q(\P^n_R),C) \cong \Hom_{c\otimes/R}(\Q(\P^n_{\mathds{Z}})_R,C) \cong \Hom_{c\otimes/\mathbb{Z}}(\Q(\P^n_{\mathds{Z}}),C) \qedhere\] \end{proof}

\section{Affine Morphisms}
 
Let $f : X \to Y$ be an affine morphism of schemes. Then $\A = f_* \O_X$ is a quasi-coherent algebra on $Y$, the morphism $f$ identifies with the structural morphism $\Spec(\A) \to Y$ and has the following universal property: There is a natural bijection
\[\Hom(Z,X) \cong \{(\phi,\sigma) : \phi \in \Hom(Z,Y) , \sigma \in \Hom_{\Alg(Z)}(\phi^* \A,\O_Z)\}\]
Here, we may identify $\Hom_{\Alg(Z)}(\phi^* \A,\O_Z) \cong \Hom_{\Alg(Y)}(\A,\phi_* \O_Z)$, but on the right hand side are not quasi-coherent in general. In this section, we want to give an analogous universal property of the tensor functor $f^* : \Q(Y) \to \Q(Z)$, following the same idea of categorification as before.

\begin{defi} Let $C$ be a cocomplete tensor category. An algebra $\A \in C$ is an object together with morphisms $u : 1_C \to \A$ and $m : \A \otimes \A \to \A$ such that the three diagrams asserting unitality, associativity and commutativity commute. The category of algebras in $C$ is denoted by $\Alg(C)$. For an algebra $\A$, an $\A$-module $M \in C$ is an object together with a morphism $\A \otimes M \to M$ such that the two diagrams asserting the compatibility with $u$ and $m$ commute. The category of $\A$-modules in $C$ is denoted by $\M(\A)$. This is again a cocomplete tensor category: The forgetful functor $\M(\A) \to C$ creates colimits. For $\A$-modules $M,N$ define $M \otimes_{\A} N$ to be the coequalizer of the two obvious maps $\A \otimes M \otimes N \to M \otimes N$, thus identifying the two actions of $\A$ on $M \otimes N$. The unit of $\M(\A)$ is $\A$. Remark that $1_C \in C$ is an algebra and $\M(1_C) = C$.\end{defi}
 
\begin{prop} \label{uEaff1} Let $C,D$ be a cocomplete tensor categories and $\A$ an algebra in $C$. Then there is an equivalence of categories
\[\Hom_{c\otimes}(\M(A),D) \cong \{(F,\sigma) : F \in \Hom_{c\otimes}(C,D), \sigma \in \Hom_{\Alg(D)}(F(\A),1_D)\}\]
\end{prop}

\begin{proof} The forgetful functor $f_* : \M(A) \to C$ has a left adjoint $f^* : C \to \M(\A)$, given by $X \mapsto X \otimes \A$, where $X \otimes \A$ is endowed with the obvious module structure. Remark that $f^*$ is cocontinuous tensor functor. Thus for every $G : \M(\A) \to D$ we get $F := G p^* : C \to D$. The multiplication $\A \otimes \A \to \A$ is a morphism $p^* \A \to 1_{\M(\A)}$, which is mapped by $G$ to a morphism $\sigma : F(\A) \to 1_D$ of algebras. Conversely, let us given a pair $(F,\sigma)$ as above. Define $G : \M(\A) \to D$ as follows\footnote{I am much obliged to Jacob Lurie who told to me how to construct the inverse functor.}: For every $M \in \M(A)$, the underlying object $p_* M \in C$ carries the structure of an $\A$-module. Thus $F(p_* M)$ is an $F(\A)$-module. Now extend scalars via $\sigma$:
\[G(M) := F(p_* M) \otimes_{F(\A)} 1_D.\]
It is easy to see that $G$ is a cocontinuous tensor functor. One checks that these constructions are inverse to each other. \end{proof}
 
\begin{cor} \label{uEaff} Let $X$ be a scheme and $\A$ a quasi-coherent algebra on $X$. Then for every cocomplete tensor category $C$ there is an equivalence of categories
{\small \[\Hom_{c\otimes}(\Q(\Spec(\A)),C) \cong \{(F,\sigma) : F \in \Hom_{c\otimes}(\Q(X),C), \sigma \in \Hom_{\Alg(C)}(F(\A),1_C)\}\]}
\end{cor}
 
\begin{proof} Algebras in $\Q(X)$ are precisely the quasi-coherent algebras $\A$ and there is an equivalence of categories $\M(\A) \cong \Q(\Spec(\A))$, since this is true for affine $X$. \end{proof}
 
\section{Tensorial schemes}

\begin{defi} An $R$-scheme $X$ is called \emph{tensorial} if for every $R$-scheme $Y$ the functor
\begin{eqnarray*}
\Hom_R(Y,X) &\to& \Hom_{c\otimes}\bigl(\Q(X),\Q(Y)\bigr) \\  f & \mapsto & f^* 
\end{eqnarray*}
is an equivalence of categories.
\end{defi}

\begin{thm} Every affine $R$-scheme is tensorial. \end{thm}

\begin{proof} This is just Corollary \ref{unimod}. \end{proof}

\begin{thm} The projective space $\P^n_R$ is tensorial. \end{thm}

\begin{proof} Let $X$ be an $R$-scheme. Then Theorem \ref{unipr} yields an equivalence of categories between $\Hom_{c\otimes/R}(\Q(\P^n_R),\Q(X))$ and the category of pairs $(\L,s)$, where $\L \in \Q(X)$ is an invertible object and $s : \O_X^{n+1} \to \L$ is a morphism such that the sequence
\[\xymatrix{(\L^{\otimes -1})^{\binom{n+1}{2}} \ar[r] &  1^{n+1} \ar[r] &  \L \ar[r] &  0}\]
is exact. According to Lemma \ref{lempr}, this just means that $\L$ is invertible in the usual sense and that $s$ is an epimorphism. Thus we arrive at a category which is known to be equivalent to $\Hom_R(X,\P^n_R)$. The resulting equivalence between $\Hom_R(X,\P^n_R)$ and $\Hom_{c\otimes/R}(\Q(\P^n_R),\Q(X))$ is $f \mapsto f^*$ by construction. \end{proof}

\begin{thm} Let $X$ be a tensorial scheme and $Y \to X$ be an affine morphism. Then $Y$ is tensorial. In particular: Every closed subscheme of a tensorial scheme is tensorial. \end{thm}

\begin{proof} Write $Y = \Spec(\A)$ for some quasi-coherent algebra $\A$ on $X$. Let $Z$ be a scheme. Then Corollary \ref{uEaff} yields:
\[\Hom_{c\otimes}(\Q(Y),\Q(Z)) \cong \{(F,\sigma) : F \in \Hom_{c\otimes}(\Q(X),\Q(Z)) ,\]
\[ \hspace{60mm} \sigma \in \Hom_{\Alg(Z)}(F(\A),\O_Z)\}\]
\[\cong \{(f,\sigma) : f \in \Hom(Z,X) ,  \sigma \in \Hom_{\Alg(Z)}(f^* \A,\O_Z)\} \cong \Hom(Z,Y).\]
Now an inspection of the equivalences shows that this is exactly the functor $f^* \mapsfrom f$. \end{proof}
 
A projective scheme over $R$ is meant to be a closed subscheme of $\P^n_R$ for some $n$. Now combine the two previous theorems:
 
\begin{thm} Every projective $R$-scheme is tensorial. \end{thm}

\begin{rem} Actually, the proof gives us for every projective $R$-scheme $X$ a universal property of $\Q(X)$. For example in case of the Fermat curve $X = V_{\P^2}(x^n+y^n-z^n)$ we get that $\Q(X)$ is the free cocomplete tensor category on an invertible object $\L$ together with three "good generators" $s_0,s_1,s_2 : 1 \to \L$ satisfying $s_0^{\otimes n} + s_1^{\otimes n} = s_2^{\otimes n}$ as morphisms $1 \to \L^{\otimes n}$.
  
Using the Plücker embedding $\mathrm{Gr}_d(R^n) \hookrightarrow \P(R^{\binom{n}{d}})$, we also get a universal property of $\Q(\mathrm{Gr}_d(R^n))$ which is similar to the one of the Grassmannian classifying locally free objects of rank $d$ with $n$ "good" generators. We plan to include details in a later work. \end{rem}

\begin{thm} For every family of tensorial schemes $(X_i)_{i \in I}$, the disjoint union $\coprod_{i \in I} X_i$ is tensorial. \end{thm}

\begin{proof} We only sketch the proof since it follows the same ideas as the previous ones. First show that for a family of cocomplete tensor categories $\{C_i\}$ the product $\prod_i C_i$ has the following universal property: Morphisms $\prod_i C_i \to D$ are given by a  "decompsition of $1_D$ into orthogonal idempotents $e_i$" and morphisms $C_i \to D_{e_i}$ for a suitably defined localization $D_{e_i}$. Apply this to $C_i=\Q(X_i)$, $D=\Q(Y)$ and use the following universal property of the disjoint union: Morphisms $Y \to \coprod_i X_i$ are given by a decomposition $Y = \coprod_i Y_i$ and morphisms $Y_i \to X_i$. \end{proof}

\begin{rem} We do not know if tensorial schemes are closed under fiber products. Namely, it is not clear how to construct $\Q(X \times_R Y)$ out of $\Q(X)$ and $\Q(Y)$. In the derived setting, this has been done (\cite{BFN}, 1.2). At least for projective schemes $X,Y$ we can show that $\Q(X \times_R Y)$ is the $2$-coproduct of $\Q(X)$ and $\Q(Y)$ in the $2$-category of all $R$-linear cocomplete tensor categories, using the Segre embedding in combination with Theorem \ref{unipr}.\end{rem}

\begin{rem} Let $X,Y$ be quasi-compact semi-separated schemes. An inspection of Jacob Lurie's proof (\cite{Lur}) shows that $\Hom(Y,X) \to \Hom_{c\otimes}(\Q(Y),\Q(X))$ is fully faithful if we restrict ourselves to tensor natural isomorphisms on the right hand side and that the essential image consists of those cocontinuous tensor functors $F : \Q(Y) \to \Q(X)$ which preserve the property of being \emph{faithfully flat}. In fact, it is enough to test that $F(p_* \O_U)$ is faithfully flat, where $p : U \to Y$ is some surjective smooth morphism with $U$ affine. This leads to the question if (faithfully) flat quasi-coherent modules may be described in the language of tensor categories. Remark that locally free quasi-coherent modules are exactly the dualizable ones, thus $F$ preserves them anyway. \end{rem}

We do not know of any scheme which is not tensorial.

\end{document}